\documentclass[12pt]{amsart}

\usepackage{amsfonts}
\usepackage{amsmath}
\usepackage{amssymb}
\usepackage{amsthm}
\usepackage[all]{xy}
\usepackage{verbatim}
\usepackage{thmtools, thm-restate}

\newtheorem{thrm}{Theorem}
\newtheorem{lemma}[thrm]{Lemma}
\newtheorem{prop}[thrm]{Proposition}
\newtheorem{cor}[thrm]{Corollary}

\theoremstyle{definition}
\newtheorem*{defn}{Definition}
\newtheorem{ex}{Example}

\DeclareMathOperator{\coker}{coker}
\DeclareMathOperator{\Hom}{Hom}

\DeclareMathOperator{\im}{im}
\DeclareMathOperator{\Tr}{Tr}
\DeclareMathOperator{\rad}{rad}
\DeclareMathOperator{\rep}{rep}
\DeclareMathOperator{\lrep}{\Lambda-rep}

\DeclareMathOperator{\md}{mod}

\DeclareMathOperator{\Sub}{Sub}
\DeclareMathOperator{\codim}{codim}
\DeclareMathOperator{\GL}{GL}

\begin{document}

\newcommand{\A}{\mathbb{A}}
\newcommand{\N}{\mathbb{N}}
\newcommand{\M}{\mathcal{M}}
\newcommand{\U}{\mathcal{U}}
\newcommand{\Z}{\mathbb{Z}}
\newcommand{\R}{\mathbb{R}}
\newcommand{\C}{\mathbb{C}}
\renewcommand{\S}{\mathcal{S}}
\newcommand{\matr}[1]{\left(\begin{matrix}#1\end{matrix}\right)}
\newcommand{\smatr}[1]{\left(\begin{smallmatrix}#1\end{smallmatrix}\right)}
\newcommand{\orb}{\mathcal{O}}
\renewcommand{\le}[1]{\leq_{\mathrm{#1}}}

\renewcommand{\r}{\frak{r}}
\renewcommand{\P}{\mathcal{P}}

\title{Degenerations of submodules and composition series}
\author{Nils Nornes}
\author{Steffen Oppermann}

\thanks{This paper will form part of the first author's ph.d-thesis, written under the supervision of Professor S. O. Smal\o. The authors thank Professor Smal\o{} for inspiration and helpful remarks.}

\begin{abstract}
Let $M$ and $N$ be modules over an artin algebra such that $M$ degenerates to $N$. We show that any submodule of $M$ degenerates to a submodule of $N$. This suggests that a composition series of $M$ will in some sense degenerate to a composition series of $N$. 

We then study a subvariety of the module variety, consisting of those representations where all matrices are upper triangular. We show that these representations can be seen as representations of composition series, and that the orbit closures describe the above mentioned degeneration of composition series.
\end{abstract}

\maketitle

\section{Introduction}

Let $k$ be an algebraically closed field, and let $\Lambda$ be a finite dimensional associative $k$-algebra with unity. 
We denote by $\md \Lambda$ the category of finite dimensional unital left modules over $\Lambda$. For  natural numbers $m$ and $n$, let $\M_{m\times n}(k)$ denote the set of $m\times n$-matrices with entries in $k$, let $\M_n(k)$ denote the $k$-algebra of $n\times n$-matrices and $\U_n(k)\subseteq\M_n(k)$ the subalgebra of upper triangular matrices. $\GL_n(k)\subseteq\M_n(k)$ denotes the general linear group, and $U_d(k)\subseteq \GL_d(k)$ denotes the subgroup of upper triangular matrices.

Fix a natural number $d$. We want to study the set of left $\Lambda$-module structures on the vector space $k^d$. We have a one-to-one correspondence between this set and the set of $k$-algebra homomorphisms from $\Lambda$ to $\M_d(k)$. If $f$ is such a homomorphism, we obtain a module structure by setting $\lambda\cdot \mathbf{v}:=f(\lambda)\mathbf{v}$ for $\lambda\in\Lambda$ and $\mathbf{v}\in k^d$. Conversely, if we have a module structure, we get a $k$-algebra homomorphism $g$ by setting $g(\lambda):=\matr{\lambda\cdot \mathbf{u}_1 &\ldots&\lambda\cdot \mathbf{u}_d}$, where $\mathbf{u}_i$ is the $i$th unit column vector. Such a homomorphism is called a \emph{$d$-dimensional representation of $\Lambda$}, and we denote the set of all $d$-dimensional representations of $\Lambda$ by $\md_d\Lambda$.

Let $\{\lambda_1,\ldots,\lambda_n\}$ be a generating set of $\Lambda$. Then a representation $\rho\in\md_d\Lambda$ is completely determined by its values on $\lambda_i$, so we can view $\md_d\Lambda$ as a subset of $\M_d(k)^n$. This subset is Zariski closed, so $\md_d\Lambda$ has the structure of an affine variety. The group variety $\GL_d(k)$ acts on $\md_d\Lambda$ by conjugation, and its orbits correspond bijectively to the isomorphism classes of modules. We can now give the definition of degeneration of modules.
\begin{defn}
Let $M$ and $N$ be $\Lambda$-modules with representations $\mu$ and $\nu$ in $\md_d\Lambda$. $M$ \emph{degenerates to} $N$ if $\nu$ lies in the closure of the $\GL_d(k)$-orbit of $\mu$. This is denoted by $M\le{deg}N$.
\end{defn}

Degeneration is a partial order on the set of isomorphism classes of $d$-dimensional modules. The \emph{codimension} of a degeneration $M\le{deg}N$, denoted $\codim(M,N)$, is the codimension of the orbit corresponding to $N$ in the closure of the orbit corresponding to $M$. The dimension of an orbit $\GL_d(k)*\mu$ can be computed by the formula $\dim \GL_d(k)*\mu=d^2-[M,M]$, where $[M,M]$ denotes the $k$-dimension of $\Hom_\Lambda(M,M)$. From that we get $\codim(M,N)=[N,N]-[M,M]$.
 
In \cite{Zwara} G. Zwara, building on earlier work of C. Riedtmann in \cite{Riedtmann}, gave a nice module-theoretic description of this partial order:

\begin{thrm}\label{RZ}
Let $M$ and $N$ be $\Lambda$-modules. Then the  following are equivalent:
\begin{enumerate}
\item $M\le{deg}N$
\item There exists a short exact sequence $0\to N\to M\oplus Z\to Z\to 0$ in $\md \Lambda$ for some $Z\in\md\Lambda$.
\item There exists a short exact sequence $0\to X\to M\oplus X\to N\to 0$ in $\md \Lambda$ for some $X\in\md\Lambda$.
\end{enumerate}
\end{thrm}

The short exact sequences in Theorem \ref{RZ} are called \emph{Riedtmann-sequences}. In this paper we will use Riedtmann-sequences of the form $0\to X\to M\oplus X\to N\to 0$, but all our results work equally well for sequences of the other form.

Now one can extend the notion of degeneration to algebras over arbitrary fields, and even over commutative artin rings, by using the existence of Riedtmann-sequences as the definition. G. Zwara showed  in \cite{Zwara1} that degeneration is a partial order also in this case. Here we define the codimension of $M\le{deg}N$ to be $[N,N]-[M,M]$ (where $[X,X]$ denotes length of $\Hom_\Lambda(X,X)$ as a $k$-module.)

One problem with the degeneration order is that in general one cannot cancel common summands, that is $X\oplus M\le{deg} X\oplus N$ does not imply $M\le{deg}N$. This led to the introduction of a new partial order called \emph{virtual degeneration} in \cite{Riedtmann}.

\begin{defn} Let $M$ and $N$ be $\Lambda$-modules. $M$ \emph{virtually degenerates to} $N$ if there exists a module $X\in\md\Lambda$ such that $X\oplus M\le{deg} X\oplus N$. This is denoted by $M\le{vdeg}N$.
\end{defn}

The following proposition gives an alternative way of describing virtual degenerations. For a proof of the proposition see \cite{SmalVal}, section 2.

\begin{prop}
Let $M$ and $N$ be $\Lambda$-modules. Then $M\le{vdeg}N$ if and only if there is some finitely presented functor $\delta:\md\Lambda \to \md k$ such that $\ell(\delta(X))=[X,N]-[X,M]$ for all $X\in\md\Lambda$.
\end{prop}
 
 If $\delta$ is such a functor, we say that the degeneration is \emph{given by} $\delta$.

In section \ref{DegSub} we will prove the following:

\begin{restatable}{thrm}{substhrm}\label{subs}
Let $M$ and $N$ be $\Lambda$-modules and let $M'\subseteq M$ be a submodule. 
\begin{enumerate}
\item\label{subdeg} If $M\le{deg}N$, then there exists a submodule $N'\subseteq N$ such that $M'\le{deg} N'$.
\item\label{subvdeg} If $M\le{vdeg}N$, then there exists a submodule $N'\subseteq N$ such that $M'\le{vdeg} N'$.
\end{enumerate}

\end{restatable}

In section \ref{trianglereps} we look at representations whose images are contained in $\U_d(k)$, which we call \emph{triangular representations}. We show that these can be viewed as representations of composition series, and then we prove the following analogue of Theorem \ref{RZ}.

\begin{restatable}{thrm}{trianglethrm}\label{RZ-triangle}
Let $\mu$ and $\nu$ be triangular $\Lambda$-representations, and let respectively $\xymatrix{M_1\ar@{^{(}->}[r]^{i_1}&\ldots\ar@{^{(}->}[r]^{i_{d-1}}&M_d}$ and  $\xymatrix{N_1\ar@{^{(}->}[r]^{j_1}&\ldots\ar@{^{(}->}[r]^{j_{d-1}}&N_d}$ be the corresponding composition series. Then $\nu\in\overline{U_d(k)*\mu}$ if and only if there exists a commutative diagram $$\xymatrix{0\ar[d]&0\ar[d]&&0\ar[d]\\
X_1\ar[r]^{h_1}\ar[d]&X_2\ar[r]^{h_2}\ar[d]&\cdots\ar[r]^{h_{d-1}}&X_d\ar[d]\\
X_1\oplus M_1\ar[r]^{\smatr{h_1&0\\0&i_1}}\ar[d]&X_2\oplus M_2\ar[r]^{\smatr{h_2&0\\0&i_2}}\ar[d]&\cdots\ar[r]^{\smatr{h_{d-1}&0\\0&i_{d-1}}}&X_d\oplus M_d\ar[d]\\
N_1\ar[d]\ar[r]^{j_1}&N_2\ar[d]\ar[r]^{j_2}&\cdots\ar[r]^{j_{d-1}}&N_d\ar[d]\\
0&0&&0}$$ with exact columns.
\end{restatable}

To study degenerations of modules, one can look at the variety of quiver representations, $\rep_{\mathbf{d}}(Q,\rho)$, instead of $\md_d\Lambda$. Let $Q$ be a quiver with vertices $Q_0=\{1,\ldots,n\}$ and arrows $Q_1$,  and let $\mathbf{d}=(d_1,\ldots,d_n)\in \N^n$. Then $\rep_{\mathbf{d}}Q=\coprod_{\alpha\in Q_1} \M_{d_{e(\alpha)}\times d_{s(\alpha)}}(k)$, where $s(\alpha)$ and $e(\alpha)$ are respectively the start and end points of the arrow $\alpha$, consists of all representations with dimension vector $\mathbf{d}$. The group variety $G_{\mathbf{d}}=\GL_{d_1}(k)\times\ldots\times \GL_{d_n}(k)$ acts on $\rep_{\mathbf{d}} Q$, and the orbits correspond to isomorphism classes. Given a set of relations $\rho$ on $Q$, $\rep_{\mathbf{d}}(Q,\rho)$ is the subvariety of $\rep_{\mathbf{d}} Q$ consisting of all representations that satisfy the relations in $\rho$. K. Bongartz showed in \cite{Bongartz} that the degeneration order we get from $\rep_{\mathbf{d}}(Q,\rho)$ is the same as the one we get from  $\md_d kQ/\langle\rho\rangle$. He also showed a deeper geometric connection between these varieties, but we will not go into that in this paper. Usually $\rep_{\mathbf{d}}(Q,\rho)$ is much smaller than  $\md_d kQ/\langle\rho\rangle$, which makes it easier to perform computations.

 In section \ref{trianglequiver} we introduce a similar smaller variety that can be used to study degenerations of composition series.

For general background on representation theory of algebras we refer the reader to \cite{ARS}. For an introduction to the topic of module degenerations, see \cite{Smalo}.

\section{Degenerations of submodules}\label{DegSub} 

In this section, let $k$ be a commutative artin ring and let $\Lambda$ be an artin $k$-algebra. All modules considered in this paper have finite length.

We first prove part \ref{subdeg} of  Theorem \ref{subs}.
 
\begin{prop}\label{subdegprop}
Let $M$ and $N$ be $\Lambda$-modules and let $M'\subseteq M$ be a submodule. If $M\le{deg}N$, then there exists a submodule $N'\subseteq N$ such that $M'\le{deg} N'$.
\end{prop}

\begin{proof}

Assume that $M\le{deg}N$ and let $M'\subseteq M$ be a submodule. Then there exists an exact sequence
$$\eta:\xymatrix{0\ar[r]&X\ar[r]^{\smatr{f\\g}}&X\oplus M\ar[r]&N\ar[r]&0}.$$
Let $X'=\{x\in X\mid gf^n(x)\in M' \quad \forall n\ge 0\}$, let $i_X:X'\to X$ and $i_M:M'\to M$ be the submodule inclusions. From the definition of $X'$, we see that $f(X')\subseteq X'$ and $g(X')\subseteq M'$. Thus, by restricting $\smatr{f\\g}$ to $X'$, we get a homomorphism $\left.\smatr{f\\g}\right|_{X'}^{X'\oplus M'}:X'\to X'\oplus M'$. Let $N'=\coker \left.\smatr{f\\g}\right|_{X'}^{X'\oplus M'}$. We then have the commutative diagram 
$$\xymatrix{&0\ar[d]&0\ar[d]&&
\\0\ar[r]&X'\ar[r]^{\left.\smatr{f\\g}\right|_{X'}^{X'\oplus M'}}\ar[d]^{i_X}&X'\oplus M' \ar[r]\ar[d]^{\smatr{i_X&0\\0&i_M}}&N'\ar[d]^\alpha\ar[r]&0
\\0\ar[r]&X\ar[r]^{\smatr{f\\g}}\ar[d]&X\oplus M\ar[r]\ar[d]&N\ar[r]&0
\\&X/X'\ar[r]^{\smatr{\overline{f}\\\overline{g}}}\ar[d]&X/X'\oplus M/M'\ar[d]&&
\\&0&0&&}$$
with exact rows and columns. Since the top row is exact we have $M'\le{deg}N'$, so it remains to show that $\alpha$ is a monomorphism. We have $$\ker \overline{f}=\{(x+X')\in X/X'\mid f(x)\in X'\}$$$$=\{(x+X')\in X/X'\mid gf^n(x)\in M' \,\forall n\ge 1\}.$$ If $(x+X')$ is a non-zero element in $\ker \overline{f}$ then $x\not\in X'=\{x\in X\mid gf^n(x)\in M' \forall n\ge 0\}$, so we must have $g(x)\not\in M'$ and hence $(x+X')\not\in\ker\overline{g}$. This means that $\ker\smatr{\overline{f}\\\overline{g}}=\ker\overline{f}\cap\ker\overline{g}=(0)$. Then by the Snake Lemma we get that $\ker\alpha=(0)$.
\end{proof}

To prove the same result for virtual degenerations, we will need the following simple lemma.

\begin{lemma}\label{subsum}
Let $X$ and $Y$ be $\Lambda$-modules, and let $M\subseteq X\oplus Y$ be a submodule. Then there exist submodules $X'\subseteq X$ and $Y'\subseteq Y$ such that $M\le{deg} X'\oplus Y'$.
\end{lemma}

\begin{proof}
Let $i:M\to X\oplus Y$ be the inclusion and $p:X\oplus Y\to X$ the projection on the first summand. We have a commutative diagram
$$\xymatrix{0\ar[r]&Y\ar[r]&X\oplus Y\ar[r]^p&X\ar[r]&0
\\0\ar[r]&\ker pi\ar@{^{(}->}[u]\ar[r]&M\ar@{^{(}->}[u]_i\ar[r]&\im pi\ar@{^{(}->}[u]\ar[r]&0}$$

 with exact rows. From the bottom row we make an exact sequence $$0\to\ker pi \to \ker pi \oplus M\to \ker pi\oplus \im pi\to 0,$$  which shows that $M\le{deg}\im pi\oplus \ker pi$.
\end{proof}
 
We can now complete the proof of  Theorem \ref{subs}.

\substhrm*
 
\begin{proof}
Part \ref{subdeg} was proved in Proposition \ref{subdegprop}, so it remains to prove part \ref{subvdeg}.

Assume that $M\le{vdeg}N$. Then there exists some $Y\in\md\Lambda$ so that $M\oplus Y\le{deg}N\oplus Y$. We have a submodule $M'\subseteq M$, and we want to find submodules $N'\subseteq N$ and $Y'\subseteq Y$ such that $M'\oplus Y'\le{deg}N'\oplus Y'$. To do so we construct two descending chains of submodules $Y=Y_1\supseteq Y_2\supseteq\ldots$ and $N=N_1\supseteq N_2\supseteq\ldots$, where $M'\oplus Y_i\le{deg} N_{i+1}\oplus Y_{i+1}$ for all $i$.

We have that $M'\oplus Y\subseteq M\oplus Y$, so by Proposition \ref{subdegprop}, there exists a submodule $Z_1\subseteq N\oplus Y$ such that $M'\oplus Y\le{deg}Z_1$. Then by Lemma \ref{subsum}, there exist submodules $N_2\subseteq N$ and $Y_2\subseteq Y$ such that $Z_1\le{deg}N_2\oplus Y_2$, so we have $M'\oplus Y_1\le{deg}N_2\oplus Y_2$.

For $i>1$, assume that we have $M'\oplus Y_{i-1}\le{deg}N_i\oplus Y_i$ and $Y_i\subseteq Y_{i-1}$. Then $M'\oplus Y_i\subseteq M'\oplus Y_{i-1}$, and we can again apply Proposition \ref{subdegprop} and  Lemma \ref{subsum} to find $N_{i+1}\subseteq N_i$ and $Y_{i+1}\subseteq Y_i$ such that $M'\oplus Y_i\le{deg}N_{i+1}\oplus Y_{i+1}$.

  Since $Y$ is artin there is some $j$ such that $Y_j=Y_{j-1}$, so we have $M'\oplus Y_j\le{deg}N_j\oplus Y_j$ and thus $M'\le{vdeg}N_j$.

\end{proof}

For a module $M$, let $\Sub M$ denote the set of submodules of $M$. The construction in the proof of Proposition \ref{subdegprop} induces a function $\phi_\eta:\Sub M\to\Sub N$. Note that if $\theta$ is a different Riedtmann-sequence for the same degeneration, the functions $\phi_\eta$ and $\phi_{\theta}$ may be different. There are several questions that are natural to ask here, for example
\begin{itemize}
\item Is $\phi_\eta$ surjective?
\item Is it injective?
\item Is the codimension of $M'\le{deg}N'$ bounded by the codimension of $M\le{deg}N$?
\item If $M\le{deg}N$ is given by a finitely presented functor $\delta$, is $M'\le{deg}N'$ given by a subfunctor of $\delta$?

\end{itemize}

As the following examples show, the answer to each of these questions is in general no.

\begin{ex}\label{Kron}

Let $k$ be a field, $Q$ the Kronecker quiver, 
$$Q:\xymatrix{1 \ar@<1ex>[r]^{\alpha} \ar[r]_{\beta}& 2},$$
 and consider the path algebra $kQ$ and the $kQ$-modules given by the quiver representations 
 $$I_2=\xymatrix{k^2 \ar@<1ex>[r]^{\smatr{1&0}} \ar[r]_{\smatr{0&1}}& k}, \qquad S_1=\xymatrix{k \ar@<1ex>[r]^{\smatr{0}} \ar[r]_{\smatr{0}}& 0},$$
   $$S_2=\xymatrix{0 \ar@<1ex>[r]^{\smatr{0}} \ar[r]_{\smatr{0}}& k},\qquad R=\xymatrix{k \ar@<1ex>[r]^{\smatr{1}} \ar[r]_{\smatr{0}}& k}$$
   $$D\Tr S_1 =\xymatrix{k^3 \ar@<1ex>[r]^{\smatr{1&0&0\\0&1&0}} \ar[r]_{\smatr{0&1&0\\0&0&1}}& k^2}.$$

  We have a degeneration $I_2\le{deg}R\oplus S_1$ given by a Riedtmann-sequence $$\eta:\xymatrix{0\ar[r]&R\ar[r]&R\oplus I_2\ar[r]&R\oplus S_1\ar[r]&0}.$$ Any $(1,1)$-dimensional regular module $R'$ is isomorphic to a submodule of $I_2$, but when $R'\not\simeq R$ the only submodule of $R\oplus S_1$ it can degenerate to is the socle. Thus we see that $\phi_\eta$ is not injective. On the other hand, there is a $k$-family of submodules of $R\oplus S_1$ that are isomorphic to $R$. But there is only one submodule of $I_2$ that can degenerate to any of these, so $\phi_\eta$ is not surjective either. 

Note also that we have $[D\Tr S_1,R\oplus S_1]-[D\Tr S_1, I_2]=1\leq[D\Tr S_1,S_1\oplus S_2]-[D\Tr S_1,R']=3$, so if $R'\le{deg}S_1\oplus S_2$ is given by a functor $\delta$, then $\delta$ can not be a subfunctor of any functor giving the degeneration $I_2\le{deg}R\oplus S_1$. 
\end{ex}

In the above example the codimension of the degeneration decreases when we go to the submodules, that is, for modules $M\le{deg}N$ and submodules $M'\le{deg}N'$ we have $\codim(M',N')\leq\codim(M,N)$. As the next example shows, this does not hold in general.

\begin{ex}
Let k be a field and $\Lambda=k[X]/(X^2)$, let $S$ be the simple $\Lambda$-module and let $p:\Lambda\twoheadrightarrow  S$ and $i:S\hookrightarrow \Lambda$ be the natural projection and inclusion.   From the Riedtmann-sequence
$$\eta:\xymatrix{0\ar[r]&S\ar[r]^{\smatr{0\\i\\0}}&S\oplus\Lambda^2\ar[r]^{\smatr{0&0&1\\0&p&0\\1&0&0}}&\Lambda\oplus S^2\ar[r]&0}$$ we see that $\Lambda^2\le{deg}\Lambda\oplus S^2$.
Let $M\subseteq \Lambda^2$ be the image of $\xymatrix{\Lambda\oplus S\ar[r]^{\smatr{1&0\\0&i}}&\Lambda^2}.$ Then $\phi_\eta(M)\simeq S^3$, and $\codim(\Lambda^2,\Lambda\oplus S^2)=2$, while  \mbox{$\codim(M,S^3)=4$.} However, for the Riedtmann-sequence 
$$\theta:\xymatrix{0\ar[r]&S\ar[r]^{\smatr{0\\0\\i}}&S\oplus\Lambda^2\ar[r]^{\smatr{0&0&p\\0&1&0\\1&0&0}}&\Lambda\oplus S^2\ar[r]&0}$$
we get $\phi_\theta(M)\simeq\Lambda\oplus S$, and then  $\codim(M,\phi_\theta(M))=0$.

\end{ex}

Applying Theorem \ref{subs} repeatedly we get a connection between the composition series of a module and the composition series of its degenerations.

\begin{cor}\label{compcor}
Let $M$ and $N$ be $\Lambda$-modules such that $M\le{deg}N$ ($M\le{vdeg}N$), and let $(0)=M_0\subseteq M_1\subseteq \cdots \subseteq M_d=M$ be a composition series of $M$. Then there is a composition series $(0)=N_0\subseteq N_1\subseteq \cdots \subseteq N_d=N$ of $N$ such that for $1\leq i\leq d$ we have $M_i\le{deg}N_i$ ($M_i\le{vdeg}N_i$). In particular,  $M_i/M_{i-1}\simeq N_i/N_{i-1}$.
\end{cor}

So given a composition series $(0)\subseteq M_1\subseteq \cdots \subseteq M_d$ of $M$ and a Riedtmann-sequence of a degeneration $M\le{deg} N$, we get a composition series $(0)\subseteq N_1\subseteq \cdots \subseteq N_d$ of $N$ that seems to be some kind of degeneration of $(0)\subseteq M_1\subseteq \cdots \subseteq M_d$. If we are working over an algebraically closed field, it seems like there should be a variety of composition series where $(0)\subseteq N_1\subseteq \cdots \subseteq N_d$ is in the orbit closure of $(0)\subseteq M_1\subseteq \cdots \subseteq M_d$. In the next section we will describe such a variety.

\section{Triangular representations}\label{trianglereps}

In this section let $k$ be an algebraically closed field, and let $\Lambda$ be a basic finite-dimensional $k$-algebra. We are going to look at the following subvariety of $\md_d\Lambda$.

\begin{defn}
We call a representation $\rho\in\md_d\Lambda$ \emph{triangular} if $\im \rho\subseteq \U_d(k)$. We denote the set of all triangular representations in $\md_d\Lambda$ by $T_d(\Lambda)$.
\end{defn}

Given any subset of $\md_d\Lambda$, an obvious question to ask is which $d$-dimensional $\Lambda$-modules have representations in the subset. As we shall see,  all $d$-dimensional $\Lambda$-modules have representations in $T_d(\Lambda)$. 

Clearly $T_d(\Lambda)$ is a closed subset of $\md_d\Lambda$, so it is an affine variety. The group variety $U_d(k)$ acts on it by conjugation. In $\md_d\Lambda$, orbits correspond to isoclasses of modules, and orbit closures can be described using Riedtmann-sequences. We are going to give a similar description of orbits and orbit closures in $T_d(\Lambda)$.

We will first show how a  triangular representation can be viewed as a representation of a module and one of its composition series. Then we show that $U_d(k)$-orbits correspond 1-1 to isoclasses of composition series. Finally, we prove Theorem \ref{RZ-triangle}, which gives an algebraic description of the orbit closures, and shows that degeneration in $T_d(\Lambda)$ is the same as the degeneration of composition series suggested by Corollary~\ref{compcor}.

Given a triangular representation $\mu=\mu_d$ we obtain a composition series in the following way: Let $M_d$ be $k^d$ with the module structure obtained from $\mu$ in the usual way. For each $i$ let $M_i$ be the submodule generated by the unit vectors $\{\mathbf{u}_1,\ldots,\mathbf{u}_{i}\}$. Then we get a representation $\mu_i$ of $M_i$ simply by deleting the rightmost column and the bottom row of each of the matrices in $\mu_{i+1}$.

Given a composition series $(0)\subseteq M_1\subseteq\ldots\subseteq M_d$ we must choose a basis of $M_d$ in order to construct a representation. Choosing the basis $\{\mathbf{x}_1,\ldots,\mathbf{x}_d\}$ such that $\mathbf{x}_i\in M_i$ for all $i$, we get a representation that is triangular.

Since triangular representations represent composition series, and all modules have composition series, it follows that all modules have triangular representations.

We say that two composition series $(0)\subseteq M_1\subseteq\cdots\subseteq M_d$ and $(0)\subseteq N_1\subseteq\cdots\subseteq N_d$ are isomorphic if $M_i\simeq N_i$ for all $i$ and these isomorphisms commute with the submodule inclusions. In $\md_d \Lambda$ the isomorphism classes correspond to $\GL_d(k)$-orbits, and we want a similar correspondence for $T_d(\Lambda)$. We will now show that the orbits of $U_d(k)$ in $T_d(\Lambda)$ correspond to isomorphism classes of composition series. 

If $\mu$ and $\nu$ are triangular representations of $(0)\subseteq M_1\subseteq\ldots\subseteq M_d$ and $(0)\subseteq N_1\subseteq\ldots\subseteq N_d$, and $\nu=g*\mu$ for some $g\in U_d(k)$, then since $g\in \GL_d(k)$ we have an isomorphism between $M_d$ and $N_d$. Let $g_d=g$ and for $1\leq i<d$ let $g_i$ be the matrix obtained from $g_{i+1}$ by deleting the bottom row and rightmost column. Then for each $i$, $g_i$ gives us an isomorphism between $M_i$ and $N_i$, and the isomorphisms commute with the inclusions, so the two composition series are isomorphic. 

Conversely, let $\mu$ and $\nu$ be triangular representations where we have an isomorphism $f$ between the corresponding composition series
$$\xymatrix{M_1\ar@{^{(}->}[r]^{m_1}\ar[d]^{f_1}&M_2\ar@{^{(}->}[r]^{m_2}\ar[d]^{f_2}&\cdots\ar@{^{(}->}[r]^{m_{d-1}}&M_d\ar[d]^{f_d}\\N_1\ar@{^{(}->}[r]^{n_1}&N_2\ar@{^{(}->}[r]^{n_2}&\cdots\ar@{^{(}->}[r]^{n_{d-1}}&N_d}.$$
The matrices of $m_i$ and $n_i$ with respect to the standard bases of $k^i$ and $k^{i+1}$ are $$A(m_i)=A(n_i)=\matr{1&\cdots&0
\\\vdots&\ddots&\vdots\\0&\cdots&1\\0&\cdots&0}.$$ It is then easy to check that the matrix of $f_i$, $A(f_i)$, will be upper triangular for each $i$, and $\nu=A(f_d)*\mu$.

A composition series of a $d$-dimensional module can also be viewed as a ``representation'' of the quiver $$\A_d:\xymatrix{1\ar[r]&2\ar[r]&\cdots\ar[r]&d},$$ but with $\Lambda$-modules and homomorphisms instead of vector spaces and linear maps. That is, we have a category $\lrep \A_d$, where the objects are series of $d$ $\Lambda$-modules and $d-1$ $\Lambda$-homomorphisms $$\xymatrix{M_1\ar[r]^{m_1}&M_2\ar[r]^{m_2}&\cdots\ar[r]^{m_{d-1}}&M_d}$$ and morphisms are commutative diagrams 
$$\xymatrix{M_1\ar[r]^{m_1}\ar[d]^{f_1}&M_2\ar[r]^{m_2}\ar[d]^{f_2}&\cdots\ar[r]^{m_{d-1}}&M_d\ar[d]^{f_d}\\N_1\ar[r]^{n_1}&N_2\ar[r]^{n_2}&\cdots\ar[r]^{n_{d-1}}&N_d},$$ and the composition series are objects in this category.
Similarly to the case of ordinary representations of $\A_d$, we have an equivalence between $\lrep \A_d$ and $\md \U_d(\Lambda)$.

We can now consider degenerations in $T_d(\Lambda)$. Clearly $\nu\in \overline{U_d(k)*\mu}$ implies $\nu\in\overline{\GL_d(k)*\mu}$, but the converse does not hold.

\begin{ex}\label{3x3.1}
Let $\Lambda=k[X]/(X^3)$ and consider $\md_3 \Lambda$. Any representation  is completely determined by its value on $X$, so we identify $\md_d\Lambda$ with the set of nilpotent $3\times3$-matrices. Let $$\mu=\matr{0&0&1\\0&0&0\\0&0&0},\qquad \nu=\matr{0&1&0\\0&0&0\\0&0&0}.$$ 
In $\md_3\Lambda$, $\mu$ and $\nu$ are in the same orbit, but in $T_3(\Lambda)$ we have
 $$\mu\in\overline{U_3(k)*\nu}=\left\{\left.\matr{0&a&b\\0&0&0\\0&0&0}\right|a,b\in k\right\},$$
 but $$\nu\not\in\overline{U_3(k)*\mu}=\left\{\left.\matr{0&0&a\\0&0&0\\0&0&0}\right|a\in k.\right\}$$

 So as a triangular representation, $\mu$ is a proper degeneration of $\nu$, even though as ordinary representations they are isomorphic.

 Let $S$ be the simple $\Lambda$-module and $Y$ the 2-dimensional indecomposable $\Lambda$-module, and let $i$ denote the inclusion $S\hookrightarrow Y$. Both $\mu$ and $\nu$ represent $S\oplus Y$, and the corresponding composition series  are 
 $$\mu:\xymatrix{0\ar@{^{(}->}[r]&S\ar@{^{(}->}[r]^{\smatr{0\\1}}&S\oplus S\ar@{^{(}->}[r]^{\smatr{1&0\\0&i}}&S\oplus Y}$$
 $$\nu:\xymatrix{0\ar@{^{(}->}[r]&S\ar@{^{(}->}[r]^{i}&Y\ar@{^{(}->}[r]^{\smatr{0\\1}}&S\oplus Y}.$$
\end{ex}

In Example \ref{3x3.1}, we have a degeneration at each level of the composition series. That is a necessary condition for having a degeneration in $T_d(\Lambda)$, but as the next example shows, it is not sufficient. 

\begin{ex}\label{3x3.2}
Keep the notation from Example \ref{3x3.1}, and let
$$\nu'=\matr{0&0&0\\0&0&1\\0&0&0}.$$
This corresponds to the composition series $$\nu':\xymatrix{S\ar[r]^{\smatr{1\\0}}&S\oplus S\ar[r]^{\smatr{1&0\\0&i}}&S\oplus Y}.$$
Between $\mu$ and $\nu'$ we have isomorphisms at each level of the composition series, but the isomorphisms do not commute with the inclusions. Thus they are not isomorphic as composition series, and $\mu$ and $\nu$ are in different $G'$-orbits. As a triangular representation, $\mu$ is a proper degeneration of $\nu'$.
Despite $\nu_i'$ being a degeneration of $\mu_i$ for each $i$, $\nu'$ is not a degeneration of $\mu$ in $T_d(\Lambda)$.
\end{ex}

 In order to get a degeneration in $T_d(\Lambda)$ we somehow need the module degenerations to ``commute'' with the inclusions. More precisely, there must be Riedtmann-sequences for the module degenerations that form a commutative diagram with the composition series.

\trianglethrm*

For Example \ref{3x3.1}, we have this diagram (where $p$ is the projection $Y\twoheadrightarrow Y/S\simeq S$):
$$\xymatrix{&0\ar[d]&0\ar[d]&0\ar[d]\\
\chi:&S\ar[r]^1\ar[d]^{\smatr{0\\1}}&S\ar[r]^i\ar[d]^{\smatr{0\\i}}&Y\ar[d]^{\smatr{0\\-p\\1}}\\
\chi\oplus\nu:&S\oplus S\ar[r]^{\smatr{1&0\\0&i}}\ar[d]^{\smatr{1&0}}&S\oplus Y\ar[r]^{\smatr{i&0\\0&0\\0&1}}\ar[d]^{\smatr{0&p\\1&0}}&Y\oplus(S\oplus Y)\ar[d]^{\smatr{0&1&p\\1&0&0}}\\
\mu:&S\ar[r]^{\smatr{0\\1}}\ar[d]&S\oplus S\ar[r]^{\smatr{1&0\\0&i}}\ar[d]&S\oplus Y\ar[d]\\
&0&0&0}$$

And for Example \ref{3x3.2}, we have this diagram:
$$\xymatrix{&0\ar[d]&0\ar[d]&0\ar[d]\\
\chi':&0\ar[r]\ar[d]&S\ar[r]^1\ar[d]^{\smatr{0\\1\\-1}}&S\ar[d]^{\smatr{0\\1\\-i}}\\
\chi'\oplus \nu':&S\ar[r]^{\smatr{0\\1\\0}}\ar[d]^1&S\oplus (S\oplus S)\ar[r]^{\smatr{1&0&0\\0&1&0\\0&0&i}}\ar[d]^{\smatr{1&0&0\\0&1&1}}&S\oplus(S\oplus Y)\ar[d]^{\smatr{1&0&0\\0&i&1}}\\
\mu:&S\ar[r]^{\smatr{0\\1}}\ar[d]&S\oplus S\ar[r]^{\smatr{1&0\\0&i}}\ar[d]&S\oplus Y\ar[d]\\
&0&0&0}$$

We now come to the proof of Theorem \ref{RZ-triangle}.

\begin{proof}
We first assume that we have a commutative diagram 
$$\xymatrix{0\ar[d]&0\ar[d]&&0\ar[d]\\
X_1\ar[r]^{h_1}\ar[d]^{\smatr{f_1\\g_1}}&X_2\ar[r]^{h_2}\ar[d]^{\smatr{f_2\\g_2}}&\cdots\ar[r]^{h_{d-1}}&X_d\ar[d]^{\smatr{f_d\\g_d}}\\
X_1\oplus M_1\ar[r]_{\smatr{h_1&0\\0&i_1}}\ar[d]&X_2\oplus M_2\ar[r]_{\smatr{h_2&0\\0&i_2}}\ar[d]&\cdots\ar[r]_{\smatr{h_{d-1}&0\\0&i_{d-1}}}&X_d\oplus M_d\ar[d]\\
N_1\ar[d]\ar[r]^{j_1}&N_2\ar[d]\ar[r]^{j_2}&\cdots\ar[r]^{j_{d-1}}&N_d\ar[d]\\
0&0&&0}$$
with exact columns, and show that this implies that $\nu\in\overline{U_d(k)*\mu}$.

The maps $i_n$ and $j_n$ are monomorphisms for all $n$, and we start by showing that $h_n$ can also be assumed to be monic. 

Let $r$ be the highest number such that $h_r$ is not monic. Let $\pi:X_r\to \im h_r$ be the natural projection and $\iota:\im h_r\to X_{r+1}$ the natural injection. We make  a new commutative diagram by replacing the $r$th column with the image of the chain complex map $(h_r,\smatr{h_r&0\\0&i_r},j_r)$:
$$\xymatrix{&0\ar[d]&0\ar[d]&0\ar[d]&\\
\cdots\ar[r]&X_{r-1}\ar[r]^{\pi h_r}\ar[d]^{\smatr{f_{r-1}\\g_{r-1}}}&\im h_r\ar[r]^{\iota}\ar[d]^\alpha&X_{r+1}\ar[r]\ar[d]^{\smatr{f_{r+1}\\g_{r+1}}}&\cdots\\
\cdots\ar[r]&X_{r-1}\oplus M_{r-1}\ar[r]^{\smatr{\pi h_{r-1}&0\\0&i_{r-1}}}\ar[d]&\im h_r\oplus M_{r}\ar[r]^{\smatr{\iota&0\\0&i_{r}}}\ar[d]^\beta&X_{r+1}\oplus M_{r+1}\ar[r]\ar[d]&\cdots\\
\cdots\ar[r]&N_{r-1}\ar[d]\ar[r]^{j_{r-1}}&N_r\ar[r]^{j_r}\ar[d]&N_{r+1}\ar[r]\ar[d]&\cdots\\
&0&0&0&}$$

The new column is a  subcomplex of a short exact sequence, so $\alpha$ is a monomorphism, and it is also a quotient of a short exact sequence, so $\beta$ is an epimorphism. Since $\dim_k(\im h_r\oplus M_r)=\dim_k\im h_r +\dim_k N_r$ it is exact. By induction, we can  construct a diagram of the desired form where all the horizontal maps are monic.

We now use a modification of Riedtmann's proof that a Riedtmann-sequence implies degeneration. 
 We want to find a family of representations $\{\nu^t\}_{t\in S}\subseteq T_d(\Lambda)$, where $S$ is an open subset of $k$, $\nu^t\in U_d(k)*\mu$ for all $t\neq 0$, and $\nu^0\in U_d(k)*\nu$. We choose a basis $B=\{\mathbf{b}_1,\ldots,\mathbf{b}_d\}$ for a complement of $\im \smatr{f_d\\g_d}$ in $X_d\oplus M_d$, in such a way that $\mathbf{b}_i\in X_i\oplus M_i$ for all $i$. Let $V$ be the span of $B$. Then we explicitly construct the modules $N_d^t$ that will correspond to the representations $\nu^t$. For each $t\in k$ we have a homomorphism $$\phi_t:\xymatrix{X_d\ar[r]^{\smatr{f_d+t\cdot 1_{X_d}\\g_d}}&X_d\oplus M_d}.$$ Let $S$ be the set of all $t\in k$ such that $\phi_t$ is a monomorphism and $\im \phi_t$ is a complement of $V$. As a vector space, $N_d^t$ is $V$. To multiply with an element in $\Lambda$, we multiply in $X_d\oplus M_d$ and project the product onto $N_d^t$ along the image of $\phi_t$. For $t\neq 0$ $\phi_t$ is a split monomorphism, so we get an isomorphism between $N_d^t$ and $M_d$. Restrictions of this yields an isomorphism between composition series, and thus we get that $\nu^t\in U_d(k)*\mu$. The map sending $t$ to $\nu^t$ is continuous, so $\nu^0$ must be in $\overline{U_d(k)*\mu}$.

To show the other implication, we embed $T_d(\Lambda)$ in $\md_a(\U_d(\Lambda))$, where $a=\frac{d(d+1)}{2}$. Let $\{\lambda_1=1_\Lambda,\lambda_2,\ldots,\lambda_n\}$ be a generating set of $\Lambda$, and let $E_{i,j}$ denote the matrix where the $j$th entry of the $i$th row is $1$, and all other entries are $0$. Then $\U_d(\Lambda)$ is generated by the matrices $$L_j=\matr{\lambda_j&\cdots&0\\\vdots&\ddots&\vdots\\0&\cdots&\lambda_j}$$ for $1\leq j\leq n$, $E_{i,i}$ for $1\leq i\leq d$ and $E_{i,i+1}$ for $1\leq i\leq d-1$. Let $\psi:T_d(\Lambda)\to \md_a\U_d(\Lambda)$ be the morphism given by the following block matrices. Here $I_n$ denotes the $n\times n$ identity matrix and $0_n$ the $n\times n$ zero matrix.
$$\psi(\mu)(L_j)=\matr{\mu_1(\lambda_j)&0&&0\\
0&\mu_2(\lambda_j)&&0\\
&&\ddots&\\
0&0&&\mu_d(\lambda_j)}$$
$$\psi(\mu)(E_{i,i})=\matr{0_1&&0&0&0\\&\ddots&&&\\0&&0_{i-1}&0&0\\0&&0&I_i&0\\0&&0&0&0}$$

$$\psi(\mu)(E_{i,i+1})=\matr{0_1&&0&0&0\\&\ddots&&&\\0&&0_{i}&I_i&0\\0&&0&0&0}$$

Clearly $\psi$ is a morphism of varieties, and $U_d(k)$-orbits in $T_d(\Lambda)$ are mapped into $\GL_a(k)$-orbits in $\md_a\U_d(\Lambda)$. Thus $\nu\in\overline{U_d(k)*\mu}$ implies $\psi(\nu)\in\overline{\GL_a(k)*\psi(\mu)}$, and by Theorem \ref{RZ} we then have an exact sequence of $\U_d(\Lambda)$-modules $0\to \hat{X}\to \hat{X}\oplus\hat{M}\to \hat{N}\to 0$. Since $\md \U_d(\Lambda)\simeq \lrep \A_d$, this gives us an exact sequence in $\lrep \A_d$, which is the commutative diagram we are looking for.

\end{proof}

\section{Smaller varieties of triangular representations}\label{trianglequiver}

When studying degeneration of modules, one can replace $\md_d \Lambda$ with a variety of quiver representations, which is usually much smaller. We want to find a similar variety smaller than $T_d(\Lambda)$. 

As in the previous section, let $k$ be an algebraically closed field, and let $\Lambda$ be a basic finite-dimensional $k$-algebra. Then there is a quiver $Q$ and a set of admissible relations $\rho$ such that $\Lambda\simeq kQ/\langle\rho\rangle$. Let $\mathbf{d}=(d_1,\ldots,d_n)$ be a dimension vector over $Q$, and let $d=d_1+\ldots+d_n$.

In the path algebra of a quiver we have some distinguished idempotents, namely the trivial paths $\{e_1,\ldots,e_n\}$. Choosing suitable idempotent matrices $A_i\in\M_d(k)$, we can identify $\rep_{\mathbf{d}}(Q,\rho)$ with the subvariety of $\md_d\Lambda$ consisting of all representations $\mu$ such that $\mu(e_i)=A_i$. We want to construct  a similar subvariety of $T_d(\Lambda)$.

Recall that a set of idempotents $\{e_1, \ldots,e_n\}$ in $\Lambda$ is called \emph{orthogonal} if $e_ie_j=0$ when $i\neq j$, and a non-zero idempotent is called \emph{primitive} if it cannot be written as a sum of two non-zero orthogonal idempotents. An orthogonal set of primitive idempotents is called \emph{complete} if it is not a proper subset of a larger orthogonal set of primitive idempotents. If the orthogonal set $\{e_1, \ldots,e_n\}$ is complete, then for any simple $\Lambda$-module $S$ we have $S\simeq e_i\Lambda/\rad e_i\Lambda$ for some $i$. The set of trivial paths in a path algebra is an example of a  complete orthogonal set of primitive idempotents. 

Let $E=\{e_1\ldots,e_n\}\subseteq \Lambda$ be an orthogonal set of primitive idempotents. We want to fix some idempotent matrices $A_i\in\U_d(\Lambda)$ and look at the subvariety of $T_d(\Lambda)$ consisting of representations $\mu$ such that $\mu(e_i)=A_i$. When we make this restriction in $\md_d\Lambda$, we go from having representations of all $d$-dimensional modules to having just those with a particular set of composition factors. When we do the same in $T_d(\Lambda)$, the sequence in which the factors occur in the composition series also matters.

\begin{prop}\label{simultri}
Let $M$ and $N$ be $d$-dimensional $\Lambda$-modules. The following are equivalent:
\begin{enumerate}
\item \label{comp}There exist composition series $(0)=M_0\subseteq M_1\subseteq \cdots \subseteq M_d=M$ and $(0)=N_0\subseteq N_1\subseteq \cdots \subseteq N_d=N$ such that $M_i/M_{i-1}\simeq N_i/N_{i-1}$ for $1\leq i\leq d$.
\item \label{triangle}For any orthogonal set $E$ of  idempotents in $\Lambda$, there exist triangular representations $\mu,\nu\in \md_d\Lambda$ of $M$ and $N$ respectively, such that $\mu(e)=\nu(e)$ for all $e\in E$.

\item \label{trianglex} There exists a complete orthogonal set $E$ of primitive  idempotents in $\Lambda$ and triangular representations $\mu,\nu\in \md_d\Lambda$ of $M$ and $N$ respectively, such that $\mu(e)=\nu(e)$ for all $e\in E$.

\end{enumerate}
\end{prop}

\begin{proof}

We first show that \ref{comp} implies \ref{triangle}. Let $E$ be an orthogonal set of idempotents. Since any idempotent can be written as a sum of primitive idempotents, and any orthogonal set can be expanded to a complete orthogonal set, we may assume that $E$ is a complete orthogonal set of primitive idempotents. 

When $d=1$, $\ref{comp}\Rightarrow\ref{triangle}$ is obvious. Assume it holds for $d=l-1$ and let $M$ and $N$ be $l$-dimensional modules satisfying \ref{comp}. Then $M_{l-1}$ and $N_{l-1}$ have triangular representations $\mu$ and $\nu$ where $\mu(e)=\nu(e)$ for all $e\in E$. We now want to construct suitable bases for $M$ and $N$. Let  
$(m_1,\ldots,m_{l-1})$ and $(n_1,\ldots, n_{l-1})$ be  bases for $M_{l-1}$ and $N_{l-1}$ corresponding to $\mu$ and $\nu$. Choose elements $m\in M\setminus M_{l-1}$ and $n\in N\setminus N_{l-1}$. Since $M/M_{l-1}$ is simple there is exactly one element $e\in E$ such that $eM/M_{l-1}\neq0$, and since $M/M_{l-1}\simeq N/N_{l-1}$ we also have $eN/N_{l-1}\neq0$. We set $m_l=em$ and $n_l=en$. Then $(m_1,\ldots,m_l)$ and $(n_1,\ldots,n_l)$ are bases for $M$ and $N$, and we let $\mu'$ and $\nu'$ be the corresponding representations. 

We now have that for any $x\in\Lambda$,
$$\mu'(x)=\left(\begin{array}{ccc|c} &&&s^x_{1}\\&\mu(x_i)&&\vdots\\&&&s^x_{l-1}\\\hline 0&\cdots&0&s^x_{l}\end{array}\right)$$ where $s^x_i\in k$. The $l-1$ first entries in row $l$ are all $0$ because $M_{l-1}\subseteq M$ is a submodule. Since $\mu(x)$ is upper triangular, $\mu'(x)$ is too. Thus we have that $\mu'$ is triangular. Similarly we see that $\nu'$ is triangular. Furthermore we have 
$$\mu'(e)=\left(\begin{array}{ccc|c} &&&0\\&\mu(e)&&\vdots\\&&&0\\\hline 0&\cdots&0&1\end{array}\right)=\nu'(e),$$
and for any other $e'\in E$ we have
$$\mu'(e')=\left(\begin{array}{ccc|c} &&&0\\&\mu(e')&&\vdots\\&&&0\\\hline 0&\cdots&0&0\end{array}\right)=\nu'(e').$$
Thus we have $\mu'(e)=\nu'(e)$ for all $e\in E$. By induction we get that $\ref{comp}\Rightarrow\ref{triangle}$.

Obviously \ref{triangle} implies \ref{trianglex}, so it remains to show that \ref{trianglex} implies \ref{comp}. Again this is obvious for $d=1$. Assume that it holds for $d=l-1$ and let $M$ and $N$ be $l$-dimensional modules satisfying \ref{trianglex}. Let $(m_1,\ldots,m_l)$ and $(n_1,\ldots,n_l)$ be the bases corresponding to $\mu$ and $\nu$. Since $\mu$ is triangular, $\{m_1,\ldots,m_{l-1}\}$ spans a submodule which we call $M_{l-1}$. We construct $N_{l-1}$ in the same way. $M_{l-1}$ and $N_{l-1}$ satisfy \ref{trianglex}, so by assumption they also satisfy \ref{comp}. All that is left to check is that $M/M_{l-1}\simeq N/N_{l-1}$. Let $x\in E$ be the idempotent with $xM/M_{l-1}\neq0$. Then we have $$xm_l\not\in M_{l-1} \Leftrightarrow \mathbf{u}_l^T\mu(x)\mathbf{u}_l=\mathbf{u}_l^T\nu(x)\mathbf{u}_l\neq0\Leftrightarrow xn_l\not\in N_{l-1}\Leftrightarrow xN/N_{l-1}\neq0,$$ which shows that $M/M_{l-1}\simeq N/N_{l-1}$.
\end{proof}

Two modules may have the same dimension vector, yet not have compatible composition series as above. Thus, when we restrict to triangular representations with fixed values on $E$, we get representations of at most one of them.

\begin{ex}
Let $Q$ be the quiver $\xymatrix{1\ar@<1ex>[r]^\alpha & 2\ar[l]^\beta}$, and let $\Lambda=kQ/(\alpha\beta,\beta\alpha)$. $\Lambda$ is generated by $\{e_1,e_2,\alpha,\beta\}$, where $e_i$ is the trivial path corresponding to the vertex $i$. Consider the quiver representations
$$M:\xymatrix{k\ar@<1ex>[r]^1&k\ar[l]^0},\qquad N:\xymatrix{k\ar@<1ex>[r]^0&k\ar[l]^1}.$$

$M$ and $N$ both have simple socles, but the socles are not isomorphic. Thus they do not satisfy statement \ref{comp} in Proposition \ref{simultri}. $\{e_1,e_2\}$ is a complete set of primitive orthogonal idempotents, so if $\mu$ and $\nu$ are representations of $M$ and $N$, and we have $\mu(e_1)=\nu(e_1)$ and $\mu(e_2)=\nu(e_2)$, then by Proposition \ref{simultri} $\mu $ and $\nu$ cannot both be triangular. 

For example, let $\mu,\nu\in\md_2\Lambda$ be the functions given by $$(\mu(e_1),\mu(e_2),\mu(\alpha),\mu(\beta))=\left(\matr{1&0\\0&0},\matr{0&0\\0&1},\matr{0&0\\1&0},\matr{0&0\\0&0}\right),$$
$$(\nu(e_1),\nu(e_2),\nu(\alpha),\nu(\beta))=\left(\matr{1&0\\0&0},\matr{0&0\\0&1},\matr{0&0\\0&0},\matr{0&1\\0&0}\right).$$
Then $\mu$ represents $M$ and $\nu$ represents $N$. We see that $\mu(e_i)=\nu(e_i)$ for $i=1,2$ but $\mu(\alpha)$ is not upper triangular, so $\mu$ is not a triangular representation. If we instead use a triangular representation of $M$, say $\mu'$ given by 
$$(\mu'(e_1),\mu'(e_2),\mu'(\alpha),\mu'(\beta))=\left(\matr{0&0\\0&1},\matr{1&0\\0&0},\matr{0&1\\0&0},\matr{0&0\\0&0}\right),$$
we get $\mu'(e_i)\neq\nu(e_i)$ (and in fact the only nonzero idempotent $e$ such that $\mu'(e)=\nu(e)$ is the identity).

\end{ex}

So we want an analogue of dimension vectors that also records the sequence of the composition factors.

\begin{defn}
The \emph{composition vector} of a composition series $(0)=M_0\subseteq M_1\subseteq\ldots\subseteq M_d$ is an element $\mathbf{c}=(c_1,\ldots,c_d)\in E\times\ldots\times E$ such that for all $i$ we have $M_i/M_{i-1}\simeq c_i\Lambda/\rad c_i\Lambda$.
\end{defn}

Now given a composition vector $\mathbf{c}$ we can construct a subvariety $T_\mathbf{c}(\Lambda)\subseteq T_d(\Lambda)$ in the following way. For $1\leq i\leq n$ let $A_i^{\mathbf{c}}$ be the diagonal $d\times d$-matrix where the $j$th element on the diagonal is $1$ if $c_j=e_i$ and $0$ otherwise. Then let $T_{\mathbf{c}}(\Lambda)=\{\mu\in T_d(\Lambda) \mid  \mu(e_i)=A_i^{\mathbf{c}}\}$. Any representation in $T_{\mathbf{c}}(\Lambda)$ represents a composition series with composition vector $\mathbf{c}$, and from Proposition \ref{simultri} we see that all composition series with this composition vector are represented in $T_{\mathbf{c}}(\Lambda)$.

We also need a suitable group variety to act on $T_{\mathbf{c}}(\Lambda)$. Since $T_{\mathbf{c}}(\Lambda)$ is a closed subset both in $T_d(\Lambda)$ and in $\md_d\Lambda$, we could use its normalizer in either $\GL_d(k)$ or $U_d(k)$. We denote these normalizers $N_{\GL_d(k)}(T_{\mathbf{c}}(\Lambda))$ and   $N_{U_d(k)}(T_{\mathbf{c}}(\Lambda))$ respectively. (For a proof that the normalizer of a closed set is a group variety, see for example \cite{groups}, Lemma 8.3.1)

If we choose $N_{\GL_d(k)}(T_{\mathbf{c}}(\Lambda))$, then the group action no longer preserves composition series. This is shown in the next example.

\begin{ex}\label{badgroup}
Let $\Lambda$ be the Kronecker algebra as in Example \ref{Kron}, and consider the modules $R$ and $R_2$ given by the following quiver representations. 
$$R:\xymatrix{k \ar@<1ex>[r]^1\ar[r]_0&k},\qquad R_2:\xymatrix{k^2 \ar@<1ex>[r]^{\smatr{1&0\\0&1}}\ar[r]_{\smatr{0&0\\1&0}}&k^2}$$
$R_2$ has triangular representations $\mu$ and $\nu$ given by 
$$\mu(e_1)=\nu(e_1)=\matr{0&0&0&0\\0&0&0&0\\0&0&1&0\\0&0&0&1}$$
$$\mu(e_2)=\nu(e_2)=\matr{1&0&0&0\\0&1&0&0\\0&0&0&0\\0&0&0&0}$$ $$\mu(\alpha)=\nu(\alpha)=\matr{0&0&1&0\\0&0&0&1\\0&0&0&0\\0&0&0&0}$$ $$\mu(\beta)=\matr{0&0&0&0\\0&0&1&0\\0&0&0&0\\0&0&0&0},\qquad \nu(\beta)=\matr{0&0&0&1\\0&0&0&0\\0&0&0&0\\0&0&0&0}$$

The corresponding composition series are $0\subseteq S_2\subseteq S_2^2\subseteq P_1\subseteq R_2$ for $\mu$ and $0\subseteq S_2\subseteq S_2^2\subseteq R\oplus S_2\subseteq R_2$ for $\nu$. They both have composition vector  $\mathbf{c}=(e_2,e_2,e_1,e_1)$, but they are not isomorphic.  $T_{\mathbf{c}}(\Lambda)$ is isomorphic to the variety of quiver representations, $\rep_{(2,2)}(Q)\simeq\M_2(k)\times \M_2(k)$, and we have $N_{\GL_4(k)}(T_{\mathbf{c}}(\Lambda))\simeq \GL_2(k)\times \GL_2(k)$. Since $\mu$ and $\nu$ both represent the module $R_2$, they are in the same $N_{\GL_4(k)}(T_{\mathbf{c}}(\Lambda))$-orbit.
\end{ex}

Example \ref{badgroup} shows that $N_{\GL_d(k)}(T_{\mathbf{c}}(\Lambda))$ is a poor choice for the group action.
The action of $N_{U_d(k)}(T_{\mathbf{c}}(\Lambda))$ on the other hand, obviously does preserve composition series. In fact, we can restate Theorem \ref{RZ-triangle} with $T_{\mathbf{c}}(\Lambda)$ in the place of $T_d(\Lambda)$.

\begin{thrm}
Let $\mathbf{c}$ be a composition vector, let  $\mu,\nu\in T_{\mathbf{c}}(\Lambda)$, and let respectively $\xymatrix{M_1\ar@{^{(}->}[r]^{i_1}&\ldots\ar@{^{(}->}[r]^{i_{d-1}}&M_d}$ and  $\xymatrix{N_1\ar@{^{(}->}[r]^{j_1}&\ldots\ar@{^{(}->}[r]^{j_{d-1}}&N_d}$ be the corresponding composition series. Then $\nu\in\overline{N_{U_d(k)}(T_{\mathbf{c}}(\Lambda))*\mu}$ if and only if there exists a commutative diagram $$\xymatrix{0\ar[d]&0\ar[d]&&0\ar[d]\\
X_1\ar[r]^{h_1}\ar[d]&X_2\ar[r]^{h_2}\ar[d]&\cdots\ar[r]^{h_{d-1}}&X_d\ar[d]\\
X_1\oplus M_1\ar[r]^{\smatr{h_1&0\\0&i_1}}\ar[d]&X_2\oplus M_2\ar[r]^{\smatr{h_2&0\\0&i_2}}\ar[d]&\cdots\ar[r]^{\smatr{h_{d-1}&0\\0&i_{d-1}}}&X_d\oplus M_d\ar[d]\\
N_1\ar[d]\ar[r]^{j_1}&N_2\ar[d]\ar[r]^{j_2}&\cdots\ar[r]^{j_{d-1}}&N_d\ar[d]\\
0&0&&0}$$ with exact columns.
\end{thrm}

The proof is the same as for Theorem \ref{RZ-triangle}, we just have to choose the basis for $V$ a little more carefully. Here we need to have $c_ib_i=b_i$ for all $i$.

\end{document}